\documentclass[10pt]{article}
\usepackage[utf8]{inputenc}
\usepackage[a4paper]{geometry}
\usepackage{graphicx, wrapfig, subcaption, setspace, booktabs}
\usepackage[T1]{fontenc}
\usepackage[english]{babel}
\usepackage{amsmath}
\usepackage{amssymb}
\usepackage{lmodern}
\usepackage{amsthm}
\usepackage[shortlabels]{enumitem}
\addto\captionsenglish{}
\usepackage{mathtools}
\usepackage{microtype}

\DeclareMathOperator{\Add}{Add}
\DeclareMathOperator{\Rem}{Rem}

\DeclareMathOperator{\C}{\mathbb{C}}

\DeclareMathOperator{\Z}{\mathbb{Z}}

\DeclareMathOperator{\el}{\ell_{\lm}}
\DeclareMathOperator{\emu}{\ell_{\mu}}

\DeclareMathOperator{\sem}{s_{\mu}}

\DeclareMathOperator{\F}{\mathbb{F}}

\DeclareMathOperator{\lm}{\lambda}

\newtheorem{innercustomthm}{Theorem}
\newenvironment{customthm}[1]
  {\renewcommand\theinnercustomthm{#1}\innercustomthm}
  {\endinnercustomthm}
\newtheorem{thm}{Theorem}[section]
\newtheorem{prop}[thm]{Proposition}

\newtheorem{lem}[thm]{Lemma}

\theoremstyle{remark}
\newtheorem{rem}[thm]{Remark}

\theoremstyle{definition}
\newtheorem{df}[thm]{Definition}

\title{ \textbf{\normalsize{ON THE MODULAR MCKAY GRAPH OF $SL_n(p)$ WITH RESPECT TO ITS STANDARD REPRESENTATION}}}

\author{\normalsize \textsc{
		MIRIAM G. NORRIS}}
		\date{}
\begin{document}

\maketitle

\begin{abstract}
Let $F$ be an algebraically closed field of prime characteristic $p$. The modular McKay graph of $G:=SL_n(p)$ with respect to its standard $FG$-module $W$ is the connected, directed graph whose vertices are the irreducible $FG$-modules and for which there is an edge from a vertex $V_1$ to $V_2$ if $V_2$ occurs as a composition factor of the tensor product $V_1 \otimes W$. We show that the diameter of this modular McKay graph is $\frac{1}{2} (p-1)(n^2-n)$.
\end{abstract} 

\section{\normalsize{INTRODUCTION}}
Let $G$ be a finite group, $F$ an algebraically closed field. For an $FG$-module $W$ the McKay graph $\mathcal{M}_F(G,W)$ is the directed graph on the set of simple $FG$-modules where there is an edge from $V_1$ to $V_2$ if $V_2$ is a composition factor of $V_1 \otimes W.$
Such graphs famously come up in the McKay correspondence which says for $G$ a finite subgroup of $SU_2(\C),$ $F=\C$ and $W$ the standard $2$-dimensional $FG$-module, $\mathcal{M}_F(G,W)$ is an affine Dynkin diagram of type $A,D$ or $E.$  
We say that a McKay graph $\mathcal{M}_F(G,W)$ is modular if the characteristic of $F$ divides the order of $|G|.$

It is a result of Burnside and Brauer \cite{10.2307/2034344} that an $FG$-module $W$ is faithful if and only if every irreducible $FG$-module occurs as a composition factor of a tensor power of $W.$ This implies that for any two irreducible $FG$-modules $V_1$ and $V_2$ there exists some integer $j$ such that $V_1 \otimes V_2^*$ has a composition factor in common with $W^{\otimes j}$.  From this we can deduce that there exists is a path from $V_1$ to $V_2$ of length $j.$ This means when $W$ is a faithful $FG$-module the graph $\mathcal{M}_F(G,W)$ is connected and it makes sense to look for results about its diameter. 

 Benkart et. al. \cite{benkart2018tensor} initiated a study of mixing times for particular random walks on modular McKay graphs. This built on work of Fulman  \cite{fulman2008convergence} who considered random walks on McKay graphs where $F=\C$. A statement about the diameter of modular McKay graphs provides a lower bound for the mixing times of the random walks considered in \cite{benkart2018tensor}.

Liebeck, Shalev and Tiep \cite{Liebeck} observed that when $G$ is a finite group and $W$ is a faithful $FG$-module there is an obvious lower bound for the diameter:
\begin{align*}
    diam \mathcal{M}_F(G,W) \geq \frac{\log(b(G)) }{\log(\dim W)}
\end{align*}
where $b(G)$ is the largest dimension of an irreducible $FG$-module.
For $F=\C,$ $G$ a finite simple group and $W$ an irreducible $FG$-module it is conjectured in \cite{Liebeck} that this bound is tight up to a constant and this is proved when $G$ has bounded rank. Motivated by this we consider the modular McKay graphs for a family of groups of unbounded rank. 
In this paper we prove the following theorem concerning the graph
$\Gamma = \mathcal{M}_F(G,V_n)$ where $F$ is a field of characteristic  $p$, $G= SL_n(p)$ and $V_n$ is the standard $n$-dimensional $FG$-module. 

\begin{thm} \label{mainresult}
The diameter of the modular McKay graph of $SL_n(p)$ with respect to its standard module $V_n$ is $\frac{1}{2}(p-1)(n^2-n).$
\end{thm}

Note that the irreducible $FG$-module of largest dimension is the Steinberg module. This has dimension $p^{\frac{1}{2}(n^2-n)}$ giving us a lower bound of $\frac{(n^2 -n )\log (p)}{2 \log(n)}$ which is much smaller than $\frac{1}{2}(p-1)(n^2-n)$ suggesting very different behaviour from the cases considered in \cite{Liebeck}.

We prove Theorem \ref{mainresult} in Section 4 by showing $\frac{1}{2}(p-1)(n^2-n)$ is both a lower and an upper bound for the diameter of $\Gamma$. We find a lower bound using basic results about $\C SL_n(\C)$-modules reviewed in Section 2. In Section 3 we use a result of Brundan and Kleshchev [Theorem V(iv),\cite{brundan2000translation}] to prove the existence of certain edges in $\Gamma.$ This allows us to show the lower bound is also an upper bound in Section 4.

\section*{\normalsize{\textit{Acknowledgements}}} The author gratefully acknowledges Martin Liebeck for suggesting the problem, guiding its progress and patiently reviewing preliminary  versions of this work. The author would also like to acknowledge  Alexander Kleshchev for an extremely helpful conversation. This work was supported by the Engineering and Physical Sciences Research Council [EP/L015234/1] through the EPSRC Centre for Doctoral Training in Geometry and Number Theory (London School of Geometry and Number Theory).

\section{\normalsize{THE MCKAY GRAPH OF $SL_n(\C)$}}

Let $\Gamma_{\C} := \mathcal{M}_{\C}(SL_n(\C),V_n)$ be the McKay graph of $SL_n(\C)$ with respect to the standard $n$-dimensional $\C G$-module $V_n.$ In this section we will review some basic facts about the ordinary representation theory of $SL_n(\C).$ This will enable us in Lemma \ref{char0result} to obtain a restriction on the edges of $\Gamma_{\C}$. 

We begin with some notation. Let $H=SL_n(\C)$ and denote by $\Phi$ the root system of type $A_{n-1}$ corresponding to some maximal torus $T \leq H.$ Fix $\Delta = \{\alpha_1,\dots,\alpha_{n-1}\}$ to be a base of simple roots in $\Phi$ and let $E$ be the Euclidean space spanned by $\Delta$. Let $\Delta^{\vee}$ denote the set of simple coroots and $\{\lm_1,\dots, \lm_{n-1}\}$ the dual basis of fundamental dominant weights relative to the inner product on $E.$ The Cartan matrix $C$ expresses the change of basis from the set of fundamental weights to $\Delta$. A weight is a vector in $E$ with integral inner product with the coroots and can be written as a integral combination of the fundamental dominant weights. Therefore if $\lm$ is a weight we can write $\lm = \sum_i m_i \lm_i$ and hence think of it as a vector of integers $m = (m_1,\dots,m_{n-1});$ this is the notation we will adopt throughout the paper. Applying the change of basis matrix $C^{-1}$ allows us to write such a weight $\lm$ as a sum of simple roots, $\lm = \sum_i c_i \alpha_i$ such that $C^{-1}m^{T}=c^{T}$ where $c = (c_1,\dots,c_{n-1}).$ We call a weight $\lm= \sum_i m_i \lm_i$ dominant if $m_i\geq 0$ for $1\leq i \leq n-1$. There exists a partial ordering on the set of weights: we write $\lm \leq \mu$ and say $\lm$ is subdominant to $\mu$ if $\mu-\lm$ has nonnegative coefficients when written as a sum of simple roots.
 
The irreducible $\C H$-modules are characterised by the set of dominant weights [see for example $\S 15$ in \cite{malle_testerman_2011}]. For a dominant weight $\lm$ we write its corresponding irreducible $\C H$-module $W_{\C}(\lm).$ These are the vertices of the graph $\Gamma_{\C}.$ The irreducible $\C H$-module corresponding to the $n-1$-tuple $(1,0,\dots,0)$ is the $n$-dimensional standard module which we denoted by $V_n.$ Furthermore the trivial module corresponds to the $n-1$-tuple $(0,\dots,0)$ which we will denote throughout as $\underline{0}.$ 
For two dominant weights $\lm,\mu$ we write $\lm \rightarrow_{\C} \mu$ if there is a directed edge in $\Gamma_{\C}$ from $W_{\C}(\lm)$ to $W_{\C}(\mu)$. The composition factors of $W_{\C}(\lm) \otimes V_n$ are known as an application of the Littlewood-Richardson rule [see for example $\S 6.1$ in \cite{fulton2013representation}].

 \begin{lem}\label{options}
Let $H=SL_n(\C)$ and let $W_{\C}(\lm)$ and $V_n$ be $\C H$-modules as defined above with $\lm = (m_1,\dots, m_{n-1}).$ Then $W_{\C}(\mu)$ is a composition factor of $W_{\C}(\lm) \otimes V_n$ if and only if $\mu$ takes the form of one of the following:
\begin{enumerate}[(a)]
    \item $(m_1+1,m_2,\dots,m_{n-1})$
    \item $(m_1,\dots, m_{i-1},m_{i}-1,m_{i+1}+1,\dots m_{n-1})$ for $1 \leq i \leq n-1.$
    \item $(m_1,\dots,m_{n-2},m_{n-1}-1).$
\end{enumerate}
\end{lem}

We now define a value associated to a dominant weight $\lm$ that will allow us to determine in Lemma \ref{char0result} a restriction on edges in $\Gamma_{\C}.$

\begin{df}\label{f}
Let $\lm$ be any dominant weight and let $c_{n-1}$ be the coefficient of the simple root $\alpha_{n-1}$ in $\lm$ written as a sum of simple roots. We define the following integer associated to $\lm,$ $$f(\lm) = n c_{n-1}.$$
\end{df}

Clearly we have $f(\underline{0})=0,$ furthermore if we let $St_p$ denote the weight $(p-1,\dots,p-1)$ then $f(St_p) = \frac{1}{2}(p-1)(n^2-n).$ Let $d(\lm,\mu)$ denote the distance from $W_{\C}(\lm)$ to $W_{\C}(\mu)$ in the graph $\Gamma_{\C}$.

\begin{lem}\label{char0result}
In the graph $\Gamma_{\C}$ the following hold.
\begin{enumerate}[1)]
    \item If $\lm \rightarrow_{\C} \mu$ then $f(\mu) \leq f(\lm) + 1.$
    \item Furthermore $d(\underline{0},St_p) = \frac{1}{2}(p-1)(n^2-n).$
\end{enumerate}
\end{lem}
 
\begin{proof}
As in the statement of Lemma \ref{options} let $\lm = (m_1,\dots,m_{n-1})$ and assume $\mu$ takes the form in options $(a),(b)$ or $(c).$ Inspecting the bottom row of the Cartan matrix of type $A_{n-1}$ we see that the coefficient of $\alpha_{n-1}$ in $\lm$ written as a sum of roots is $\sum_{i=1}^{n-1} \frac{i}{n}m_i.$ Therefore $f(\mu) = f(\lm)+1$ if $\mu$ takes the form $(a)$ or $(b)$ and $f(\mu) = f(\lm)-(n-1)$ if $\mu$ takes the form $(c).$ In particular $f(\mu) \leq f(\lm)+1$ completing the proof of 1).

If follows from 1) that $d(\underline{0},St_p) \geq f(St_p) = \frac{1}{2}(p-1)(n^2-n)$ hence to show 2) it will suffice to find a path of length $\frac{1}{2}(p-1)(n^2-n)$ from $W_{\C}(\underline{0})$ to $W_{\C}(St_p).$ We start at the vertex $W_{\C}(\underline{0})$ and call this stage $0.$ At each stage $j>0$ we begin with a vertex $W_{\C}(\lm)$ where $\lm$ has $p-1$'s in the last $j-1$ entries and $0$'s in all other entries. As a first step we move along the edge corresponding to option $(a)$ in Lemma \ref{options} and then we repeatedly move along the edges described by option $(b)$ for $i=1$ to $i=n-j$. We repeat this step $p-1$ times resulting in a weight with $p-1$ in the $n-j$ entry. Terminating this process after $n-1$ stages will result in a weight in which every entry is $p-1$ as required.  

We illustrate the algorithm with the example $n=5, p=2$.  
\begin{align*}
         W_{\C}(0,0,0,0) &\xrightarrow{\textit{stage 1}} W_{\C}(1,0,0,0) \rightarrow W_{\C}(0,1,0,0) \rightarrow W_{\C}(0,0,1,0) \rightarrow W_{\C}(0,0,0,1)\\  &\xrightarrow{\textit{stage 2}}  W_{\C}(1,0,0,1)   \rightarrow W_{\C}(0,1,0,1) \rightarrow W_{\C}(0,0,1,1) \\ &\xrightarrow{\textit{stage 3}}  W_{\C}(1,0,1,1) \rightarrow W_{\C}(0,1,1,1) \\&\xrightarrow{\textit{stage 4}}  W_{\C}(1,1,1,1). 
\end{align*}
For each $0<j \leq n-1$ the $j$th stage requires passing along $(n-j)(p-1)$ edges and therefore the length of the path is $\frac{1}{2}(p-1)(n^2-n)$ as required. 

\end{proof} 

\section{\normalsize{EDGES IN THE MODULAR MCKAY GRAPH OF $SL_n(p)$}}
 
 In this section we will review some facts and introduce some notation that will allow us to prove in Lemma \ref{paths} the existence of some edges in the modular McKay graph of $SL_n(p).$
  
 Let $K = \overline{\F}_p$ and denote by $SL_n$ the algebraic group $SL_n(K).$ Let $G= SL_n(p)$ the fixed point group under the standard Frobenius map on $SL_n$ that sends $(a_{i,j}) \mapsto (a_{i,j}^p).$ Recall from the introduction that $\Gamma= \mathcal{M}_K(G,V_n)$ is the modular McKay graph of $G=SL_n(p)$ with respect to the standard $n$-dimensional $KG$-module $V_n$. We briefly describe the relationship between $\C SL_n(\C)$-,$ K SL_n$- and $KG$-modules in the following discussion.
 
 As previously let $W_{\C}(\lm)$ be the irreducible $\C H$-module indexed by a dominant weight $\lm.$ There exists a $\Z$-form $W_{\Z}(\lm)$ contained in $W_{\C}(\lm)$ such that $SL_n$ acts on $W(\lm) = W_{\Z}(\lm) \otimes K$ [see for example $\S 1.5$ in \cite{doty1985submodule}]. We call the $KSL_n$-module $W(\lm)$ the Weyl module corresponding to $\lm$. For each dominant weight $\lm$ the Weyl module $W(\lm)$ has a unique maximal submodule $M(\lm)$ such that the quotient $V(\lm)= W(\lm) / M(\lm)$ is an irreducible $KSL_n$-module. The composition factors of $M(\lm)$ are irreducible $KSL_n$-modules $V(\mu)$ such that $\mu < \lm$ under the dominance relation described in the Section $2$. The irreducible $KG$-modules are the restrictions of $V(\lm)$ for $p$-restricted weights $\lm.$ For simplicity we will abuse notation and refer to $V_n$ as the $\C H$-module $W_{\C}(1,0,\dots,0),$ the corresponding Weyl module $W(1,0,\dots,0)$ and also its restriction to $G.$ Note that $W(1,0,\dots,0)$ = $V(1,0,\dots,0)$ as there are no dominant weights subdominant to $(1,0,\dots,0).$
 
 The vertices of $\Gamma=\mathcal{M}_K(G,V_n) $ are therefore parametrised by the set of $p$-restricted dominant weights of length $n-1.$ For two $p$-restricted weights $\lm,\mu$ we write $\lm \rightarrow \mu$ if there is an edge in $\Gamma$ from $V(\lm)|_{G}$ to $V(\mu)|_G$ i.e. if $V(\mu)|_G$ is a composition factor of $V(\lm)|_{G} \otimes V_n.$
 
 Let us now make some definitions that are necessary to state Theorem B (iv) from \cite{brundan2000translation} which we then use to prove Lemma \ref{paths}. To a dominant weight $\lm$ we can associate a partition $\Tilde{\lm}$ of length $n$ such that $\lm_i = \Tilde{\lm}_i-\Tilde{\lm}_{i-1}$ and $\Tilde{\lm}_n=0.$ Let $\epsilon_i$ be the $n$-tuple whose entries are all $0$ except for a $1$ in position $i$. 
 
 We say $i$ is $\Tilde{\lm}$-\textit{addable} if $\Tilde{\lm}+\epsilon_i$ is a partition, that is $\Tilde{\lm}_{i+1} \leq \Tilde{\lm}_{i}+1 \leq \Tilde{\lm}_{i-1}$. We denote by $\Add(\Tilde{\lm})$ the set of $\Tilde{\lm}$-addable $i.$ Similarly we say $i$ is $\Tilde{\lm}$-\textit{removable} if $\Tilde{\lm}-\epsilon_i$ is a partition and denote by $\Rem(\Tilde{\lm})$ the set of $\Tilde{\lm}$-removable $i.$ 
 
 If $i$ is $\Tilde{\lm}$-addable then we say it is \textit{conormal} for $\Tilde{\lm}$ if there is an injection $g$ from the set $$ R_i = \{k \in \Rem(\Tilde{\lm}): 1\leq k<i, \hspace{0.3cm} \Tilde{\lm}_i+1-i \equiv \Tilde{\lm}_k-k \mod p\}$$ into the set $$ A_i = \{ k' \in \Add(\Tilde{\lm}):1\leq k'<i, \hspace{0.3cm} \Tilde{\lm}_i + 1-i \equiv \Tilde{\lm}_{k'}+1-k' \mod p\}$$ such that $g(k)>k$ for all $k \in R_i.$

Partitions $\Tilde{\mu}$ of length $n$ parametrise the irreducible $KGL_n(K)$-modules which we denote by $V'(\Tilde{\mu})$. Furthermore $GL_n(K)$ is the central product of  $ SL_n $ and $K^{\times}$ where the action of $K^{\times}$ on $V'(\Tilde{\mu})$ is the $\Tilde{\mu}_n$th power of the determinant. Restricting $KGL_n(K)$-modules to the subgroup $SL_n$ we may ignore the action $K^{\times}$ which is indicated in the partition by the $n$th entry. For a partition $\Tilde{\mu}$ of length $n$ we have $V'(\Tilde{\mu}_1,\Tilde{\mu}_2,\dots,\Tilde{\mu}_n)|_{SL_n}= V(\mu')$ where the partition corresponding to $\mu'$ is $(\Tilde{\mu}_1-\Tilde{\mu}_n,\Tilde{\mu}_2-\Tilde{\mu}_n,\dots,\Tilde{\mu}_{n-1}-\Tilde{\mu}_n).$ In particular we note $V'(1,0,\dots,0)|_{SL_n} = V_n.$ The following theorem is a re-statement of Theorem B(iv) in \cite{brundan2000translation}.

\begin{thm} \label{bk}
For a partition $\Tilde{\mu}$ of length $n,$ let $V'(\Tilde{\mu})$ be the corresponding irreducible $K GL_n(K)$-module. If $i$ is conormal for $\Tilde{\mu}$ then $V'(\Tilde{\mu}+\epsilon_i)$ occurs as a composition factor of $V'(\Tilde{\mu})\otimes V'(1,0,\dots,0).$
\end{thm}

The following result describes some edges in $\Gamma.$

\begin{lem}\label{paths}
If $\lm$ and $\mu$ are $p$-restricted weights satisfying one of the following three conditions, then $\lm \rightarrow \mu.$
\begin{enumerate}[1)]
    \item $\lm$ is any $p$-restricted weight and $\mu_1$ is the unique element of the set $\{1,\dots,p-1\}$ such that $\mu_1 \equiv \lm_1 + 1 \mod p-1,$ and $\mu_i=\lm_i$ for $i \neq 1.$
    \item $\lm$ is a $p$-restricted weight whose first nonzero entry is $\lm_s$ for some $1\leq s < n-1$, $\mu_s=\lm_s-1$ and $\mu_{s+1}$ is the  unique element of the set $\{1,\dots,p-1\}$ such that $\mu_{s+1} \equiv \lm_{s+1} + 1 \mod p-1,$ and $\mu_i=\lm_i$ for $i \neq s, s+1.$
    \item $\lm$ is a $p$-restricted weight whose first nonzero entry is $\lm_{n-1}$,  $\mu_{n-1}=\lm_{n-1}-1$ and $\mu_i = \lm_i =0 $ for $i \neq n-1.$
\end{enumerate}
\end{lem}

\begin{proof}
To show $\lm \rightarrow \mu$ it is enough to show that there is some weight $\mu'$ such that $V(\mu')$ is a composition factor of $V(\lm) \otimes V_n$ and $V(\mu)|_G$ is a composition factor of $V(\mu')|_G.$ By Theorem \ref{bk} if $i$ is conormal for $\Tilde{\lm}$ then $V'(\Tilde{\lm}+ \epsilon_i)|_{SL_n}$ is a composition factor of $V'(\Tilde{\lm})|_{SL_n} \otimes V_n$.  Furthermore as discussed above $V'(\Tilde{\lm})|_{SL_n} \otimes V'(1,0,\dots,0)_{SL_n} = V(\lm) \otimes V_n.$ Therefore we will look for composition factors of $V'(\Tilde{\lm}+ \epsilon_i)|_{SL_n}$ for $i$ conormal for $\Tilde{\lm}.$

We begin by defining two possible values of $i$ that are always conormal for $\Tilde{\lm}.$ Grouping consecutive entries of $\Tilde{\lm}$ with the same values we can write $\Tilde{\lm}$ in block form. That is, we write $\Tilde{\lm}= (\Tilde{\lm}_1^{a_1},\Tilde{\lm}_{1+a_1}^{a_2}, \dots)$ where the first $a_1$ entries all have value $\Tilde{\lm}_1$ and the next $a_2$ entries have value $\Tilde{\lm}_{1+a_1} \neq \Tilde{\lm}_1$ and so on. We claim that $1$ and $1+a_1$ are always conormal for $\Tilde{\lm}.$ To see this for $1$ is trivial since there can be no  $j \in Rem(\Tilde{\lm})$ such that $j<1.$ To see that $1+a_1$ is conormal note that the first $a_1$ entries have the same value, so $a_1$ is the only element of $Rem(\Tilde{\lm})$ strictly smaller than $1+a_1.$ However $\Tilde{\lm}_{a_1}$ and $\Tilde{\lm}_{1+a_1}$ are strictly less than $p$ and $\Tilde{\lm}_{a_1} \neq \Tilde{\lm}_{1+a_1},$ so $\lm_{a_1} \not\equiv \lm_{a_1+1} \mod p.$ Therefore the condition for $1+a_1$ to be conormal is met.

We now show that if $1)$ holds then $\lm \rightarrow \mu$. By the remarks above $V'(\Tilde{\lm}+ \epsilon_1)|_{SL_n}$ is a composition factor of $V(\lm)\otimes V_n$.  Since $\Tilde{\lm}_{n}=0$ there exists some weight $\mu'$ such that $V'(\Tilde{\lm}+ \epsilon_1)|_{SL_n} = V(\mu').$ In particular $\mu'_1 = \lm_1+1$ and $\mu'_i=\lm_i$ for $i \neq 1.$ Since $\lm$ is $p$-restricted $\mu'$ is $p$-restricted unless $\lm_1=p-1.$ If $\mu'$ is $p$-restricted then $\mu'=\mu$ and we have $\lm \rightarrow \mu.$ If $\lm_1=p-1$ then $\mu'_1=p$ and by Steinburg's tensor product theorem we can write $V(\mu') = V(0,\mu'_2,\dots,\mu'_{n-1}) \otimes V(1,0,\dots,0)^{( p)}.$ Restricting to $G$ this gives us the $KG$-module $(V(0,\mu'_2,\dots,\mu'_{n-1}) \otimes V(1,0,\dots,0))|_{G}.$ Furthermore we note $\mu =(1, \mu'_2, \dots, \mu'_{n-1})$ and therefore $V(\mu)|_G$ is a composition factor of $(V(0,\mu'_2,\dots,\mu'_{n-1}) \otimes V(1,0,\dots,0))|_{G}$ and hence $\lm \rightarrow \mu.$  

Now suppose $2)$ holds. Since $\Tilde{\lm}_i=\Tilde{\lm}_{i+1}$ for $1 \leq i \leq a_1$ but $\Tilde{\lm}_{a_1}\neq \Tilde{\lm}_{a_1+1}$ we have that $s=a_1.$ Therefore $a_1+1<n$ and hence there exists some dominant weight $\mu'$ such that $V(\mu') = V'(\Tilde{\lm}+ \epsilon_{a_1+1})|_{SL_n}.$ In particular $\mu'_{a_1} = \lm_{a_1} -1$, $\mu'_{a_1+1} = \lm_{a_1+1} +1$ and $\lm_i = \mu_i$ for $i \neq a_1,a_1+1.$ Applying the same arguments as above we see that if $\lm_s<p-1$ then $\mu'$ is $p$-restricted and $\mu = \mu'$ and therefore $\lm \rightarrow \mu$. Otherwise if $\lm_s =p-1$ we have $\mu=(\mu'_1,\dots,\mu'_{s}-1,1,\mu'_{s+2},\dots, \mu'_{n-1})$ and therefore $V(\mu)|_G$ is a composition factor of $V(\mu')|_G$ implying $\lm \rightarrow \mu.$

Finally we suppose $3)$ holds. As remarked above if $s=n-1$ is the position of the first non-zero entry in $\lm$ then $a_1+1=n.$ Therefore $\Tilde{\lm} + \epsilon_{a_1+1}=(\Tilde{\lm}_{1},\Tilde{\lm}_{2},\dots, \Tilde{\lm}_{n-1}, 1).$ Taking $\mu'$ to be the weight such that $V(\mu') = V'(\Tilde{\lm} + \epsilon_{a_1+1})|_{SL_n}$ we have $\Tilde{\mu'}=(\Tilde{\lm}_1-1,\Tilde{\lm}_2 -1,\dots, \Tilde{\lm}_{n-1}-1,0).$ Therefore $\mu'_{n-1}=\lm_{n-1}-1$ and $\mu'_i = \lm_i =0$ for $1 \leq i <n-1$ and hence $\mu=\mu'.$ Therefore we have $\lm \rightarrow \mu$ completing the proof. 
\end{proof}

\begin{rem}\label{moving}
In the proofs in the next section we will refer to moving along an edge from $V(\lm)|_G$ to $V(\mu)|_G$ where $\lm$ and $\mu$ are as in part 1) of Lemma \ref{paths}, as adding $1$ to the first entry of $\lm.$ Likewise for $\lm$ and $\mu$ as in part 2) we will refer to this move as clearing a $1$ from the first nonzero entry of $\lm$ and adding it to the next entry. We will also refer to this process as moving $1$ along. Finally for $\lm$ and $\mu$ as in part 3) of Lemma \ref{paths} we will refer to this move as clearing $1$ from the entry in position $n-1$ of $\lm.$
\end{rem}

\section{\normalsize{THE DIAMETER OF $\Gamma$}}

In this section we will combine results from Section 2 and Section 3 to prove Theorem \ref{mainresult}. Recall that $G=SL_n(p),$ $K$ is an algebraically closed field of characteristic $p$ and Theorem \ref{mainresult} concerns the McKay graph $\Gamma=\mathcal{M}_K(G,V_n)$ where $V_n$ is the standard $KG$-module.

\begin{customthm}{1.1}
The diameter of the modular McKay graph of $SL_n(p)$ with respect to its standard module $V_n$ is $\frac{1}{2}(p-1)(n^2-n).$
\end{customthm} 

The proof of Theorem \ref{mainresult} follows immediately  from Lemma \ref{short} in which we show $\frac{1}{2}(p-1)(n^2-n)$ is a lower bound for the diameter of $\Gamma$ and Lemma \ref{long} in which we show this is also an upper bound. For a weight $\lm$ recall the integer $f(\lm)$ from Definition \ref{f} in Section 2. In order to prove Lemma \ref{short} we will need the following result. 

\begin{prop}\label{tool}
 Let $\mu$ be a $p$-restricted weight and $\mu'$ a dominant weight for $SL_n(K)$. If $V(\mu)|_G$ is a composition factor of $V(\mu')|_G$ then $f(\mu) \leq f(\mu').$
\end{prop}
 
\begin{proof}
For a weight $\lm = (\lm_1,\dots,\lm_{n-1})$ let $S(\lm) = \sum_{i=1}^{n-1} \lm_i$. Note that if $\lambda$ and $\nu$ are weights such that $\lm - \nu$ written as a sum of roots corresponds to the $n-1$-tuple $(c_1,\dots,c_{n-1})$ then applying the Cartan matrix of type $A_{n-1}$ we see $S(\lm-\nu) = c_1+c_{n-1}$. Therefore if $\nu$ is subdominant to  $\lm$ then $S(\nu) \leq S(\lm)$. We will prove Proposition \ref{tool} by induction on $S(\mu')$. Let $\mu,\mu'$ be as in the statement of the Proposition and observe that if $\mu'$ is $p$-restricted then $\mu = \mu'$ and hence $f(\mu) = f(\mu').$ We will now prove the inductive step. Suppose for all weights $\gamma$ such that $S(\gamma)<S(\mu')$ the proposition holds. By Steinberg's tensor product theorem we may write $V(\mu') = \bigotimes_i V(\nu_i)^{(p^{i-1})}$ and therefore $V(\mu')|_{G} = \bigotimes_i V(\nu_i)|_{G}.$ Since we assume that $V(\mu)|_G$ is a composition factor of $V(\mu')|_{G} = \bigotimes_i V(\nu_i)|_G$ there exists a dominant weight $\gamma$ such that $V(\gamma)$ is a composition factor of $\bigotimes_i V(\nu_i)$ and $V(\mu)|_G$ is a composition factor of $V(\gamma)|_G.$ Note that $\gamma \leq \sum_i \nu_i$ and therefore $f(\gamma) \leq f(\sum_i \nu_i) \leq f(\mu')$ and $S(\gamma) \leq S(\sum_i \nu_i)$. Since $\mu'$ is not $p$-restricted we see $S(\sum_i \nu_i) = \sum_i S(\nu_i) < S(\mu')$. Therefore by our inductive assumption $f(\mu) \leq f(\gamma)$. Since $f(\gamma) \leq f(\mu')$ this completes the inductive step.  
\end{proof}

For two $p$-restricted dominant weights, $\lm$ and $\mu$ let $d_{\Gamma}(\lm,\mu)$ denote the the distance from $V(\lm)|_G$ to $V(\mu)|_G$ in $\Gamma.$ Furthermore let $St_p$ denote the weight $(p-1,\dots,p-1)$ of $SL_n(K).$

\begin{lem}\label{short}
In the graph $\Gamma$ we have $d_{\Gamma}(\underline{0},St_p) = \frac{1}{2}(p-1)(n^2-n).$ Hence this is a lower bound on the diameter of $\Gamma.$
\end{lem}

\begin{proof}
As in the proof of Lemma \ref{char0result} we will first show that if $\lm$ and $\mu$ are two $p$-restricted weights such that $\lm \rightarrow \mu$ then $f(\mu) \leq f(\lm) + 1$. This provides the desired lower bound on $d(\underline{0},St_p)$ since $f(\underline{0}) =0$ and $f(St_p)= \frac{1}{2}(p-1)(n^2-n).$ Since $\lm \rightarrow \mu$ there exists a dominant weight $\mu'$ such that $V(\mu')$ is a composition factor of  $V(\lm) \otimes V_n$ and $V(\mu)|_G$ is a composition factor of $V(\mu')|_G.$ By Proposition \ref{tool}, $f(\mu) \leq f(\mu')$ and so it remains to show that $f(\mu') \leq f(\lm) +1$. Observe that the $KSL_n$-module $W(\lm) \otimes V_n$ must contain the composition factors of $V(\lm) \otimes V_n,$ one of which is $V(\mu')$. Therefore $\mu'$ is subdominant to some weight $\eta$ where the composition factors of the Weyl module $W(\eta)$ are composition factors of $W(\lm) \otimes V_n.$ For such an $\eta$ we have $\lm \rightarrow_{\C} \eta$ and hence by Lemma \ref{char0result}, $f(\eta) \leq f(\lm)+1.$ Finally since $\mu' \leq \eta$ we have $f(\mu')\leq f(\lm) +1.$

It now remains to show that a path of length $\frac{1}{2}(p-1)(n^2-n)$ from $V(\underline{0})|_G$ to $V(St_p)|_G$ exists. Recall the path described in the proof of Lemma \ref{char0result}. Note that any two consecutive vertices in that path are irreducible $\C H$-modules whose highest weights $\lm$ and $\mu$ are $p$-restricted. Furthermore each such $\lm$ and $\mu$ satisfy one of the first two conditions in Lemma \ref{paths}. Therefore $V(\lm)|_G$ and $V(\mu)|_G$ are vertices in $\Gamma$ and by Lemma \ref{paths}, $\lm \rightarrow \mu.$ This describes a path in $\Gamma$ from $V(\underline{0})|_G$ to $V(St_p)|_G$ of length $\frac{1}{2}(p-1)(n^2-n)$. 

\end{proof}

We now complete the proof of Theorem \ref{mainresult} by showing that $\frac{1}{2}(p-1)(n^2-n)$ is also an upper bound for the diameter of $\Gamma.$ For convenience we will refer to the vertices of $\Gamma$ by the $p$-restricted dominant weights that parametrise them. 

\begin{lem}\label{long}

For any two $p$-restricted weights $\lm,\mu$ we have $d_{\Gamma}(\lm,\mu) \leq \frac{1}{2}(p-1)(n^2-n).$ Hence this is an upper bound for the diameter of $\Gamma$.

\end{lem}

\begin{proof}
Let $\lm$ and $\mu$ be distinct $p$-restricted weights. We will show there is a path from $V(\lm)|_G$ to $V(\mu)|_G$ of length no more than $\frac{1}{2}(p-1)(n^2-n).$ In order to describe such a path we make the following definitions. Let $M$ be the set of vertices in the path detailed in the proof of Lemma \ref{short} from $V(\underline{0})|_G$ to $V(St_p)|_G.$ For any $p$-restricted weight $\mu$ we define the following integer \begin{align*}
    \emu: = \begin{cases} 0 \hspace{2cm} & \hspace{1cm} \mu = St \\ n \hspace{2cm} & \hspace{1cm} \mu = \underline{0} \\
     \max\{x \in \Z | 1 \leq x \leq n-1, \hspace{0.2cm} \mu_{x} <p-1\} & \hspace{1cm} \text{otherwise}.
    \end{cases}
\end{align*}
If $\mu \in M$ we say $\mu \in \widetilde{M}$ if all entries before position $\emu$ are $0.$ We define the integer \begin{align*}
    \sem: = \begin{cases} 0  \hspace{2cm} & \hspace{1cm} \mu \in \widetilde{M} \\ \max \{x \in \Z | 1 \leq x <\emu, \hspace{0.2cm} \mu_{x}>0 \} & \hspace{1cm} \text{otherwise}.
    \end{cases}
\end{align*}

To a $p$-restricted weight $\mu$ we associate an element of $M(\mu) \in M$ defined as follows. If $\mu \in M$ let $M(\mu)=\mu$, otherwise let $M(\mu):= (0,\dots,0,1,0,\dots,0,\mu_{\emu},p-1,\dots,p-1)$ where the $1$ is in position $\sem$ and the $\mu_{\emu}$ is in position $\emu.$ By Lemma \ref{paths} there exists a path from $M(\mu)$ to $\mu$ as follows. If $\mu \in M$ then the path is trivial so we may assume $1 \leq \sem <\emu$. We add $1$ to the first entry, move it along repeatedly until we add $1$ in position $\sem$ and repeat this $\mu_{\sem}-1$ times. This results in a weight in which the entry in position $\sem$ is $\mu_{\sem}.$ We then add $1$ to the first entry, move it along until it is in position $\sem-1$ and repeat $\mu_{\sem-1}$ times. This results in a weight in which the entry in position $\sem-1$ is $\mu_{\sem-1}$. Continue in this way until the resulting weight is $\mu.$ The total number of steps taken in this path is  $$(\mu_{\sem}-1)\sem+ \sum_{i=1}^{\sem-1} i \mu_i.$$ 

We will now describe a path from $\lm$ to $M(\mu)$ for cases (i) $\el > \emu$ and (ii) $\el \leq \emu.$ The path exists by Lemma \ref{paths} where we are using the terminology from Remark \ref{moving}.

Suppose $\lm \neq \underline{0}$ and $\lm$ and $\mu$ are as in (i). There is a path as follows. Starting at $\lm$
 add 1 to the first entry $\lm_0$ times where $\lm_0$ is the unique element of the set $\{0,\dots,p-2\}$ such that \begin{equation*}
    \lm_0 + \sum_{i=1}^{\el}\lm_i \equiv 0 \mod p-1.
\end{equation*} The resulting weight is one in which the entry in the first position is $0$ if $\lm_0=\lm_1=0$ and the unique integer $S_1 \in \{1,2,\dots,p-1\} $ such that $S_1 \equiv \lm_0+\lm_1 \mod p-1$ otherwise. Then clear $1$ from the first entry $S_1$ times resulting in a weight in which the first entry is $0.$ Continue in this way; for $1<j\leq \ell_{\lm}$ let $j-1$ be the position of the first nonzero entry, clear $1$ from position $j-1$ and add it to position $j$ repeatedly until the resulting weight has $0$ in position $j-1$. The entry in the position $j$  of this weight will be $0$ if $\sum_{i=0}^{j} \lm_i= 0$ and the unique integer $S_j \in \{1,2,\dots,p-1\}$ such that $S_j \equiv \sum_{i=0}^{j} \lm_i \mod p-1$ otherwise. Note that it follows from the definition of $\lm_0$ that if $\sum_{i=0}^{\el} \lm_i \neq 0$ then in the final step clearing $1$'s from position $\el-1$ until it is $0$ will result in a weight whose entry in position $\el$ is $p-1$. Therefore we have described a path from $\lm$ to a weight in which all entries before position $\el$ are $0$ and the entry in position $\el$ is $0$ if $\sum_{i=0}^{\el} \lm_i= 0$ and $p-1$ otherwise. All entries in position greater than $\el$ in the resulting weight will be $p-1.$ This path takes $\lm_0 + \sum_{i=1}^{\el-1}S_i$ steps. Then add 1 to the first entry and move it along to position $\el-1$ and repeat $p-1$ times. Continue to add $1$ and move it along to the furthest position whose entry is not $p-1$ until the entries in all positions after $\emu$ are $p-1.$ We then add $1$ and move it along to position $\emu$ and repeat $\mu_{\emu}$ times. Finally add $1$ and move it to position $\sem$. This path takes a further $\sum_{i=\emu+1}^{\el-1}i(p-1) + \emu \mu_{\emu} + s_{\mu}$ steps and the resulting weight is $M(\mu).$

Letting $X$ denote the number of edges in the path above from $\lm$ to $M(\mu)$ and then to $\mu$ we have
\begin{align*}
 X &=  \lm_0 + \sum_{i=1}^{\el-1}S_i + \sum_{i=\emu+1}^{\el-1}i(p-1) + \emu \mu_{\emu} + \sum_{i=1}^{\sem} i \mu_i \\
 &= \lm_0 + \sum_{i=1}^{\el-1}S_i + \sum_{i=\emu+1}^{\el-1}i(p-1) + \sum_{i=1}^{\emu} i \mu_i.
\end{align*} Note that $\lm_0$ and each of the $S_i$ are all no greater than $p-1$  and $\el<n$ which implies $\lm_0 + \sum_{i=1}^{\el-1}S_i \leq (n-1) (p-1).$ Furthermore since each $\mu_i \leq p-1$ we have $X \leq \sum_{i=1}^{n-1}(p-1)i = \frac{1}{2}(p-1)(n^2-n).$

Suppose now that $\lm=\underline{0}.$ The path from $\lm$ to $M(\mu)$ is the path described in Lemma \ref{short} since $M(\mu)\in M.$ If $\mu \in M$ then $d_{\Gamma}(\lm,\mu) \leq \frac{1}{2}(p-1)(n^2-n)$ by Lemma \ref{short}. Therefore we may assume $\mu \not\in M$ in which case, looking at the definition of $M(\mu),$ we see that the path from $\underline{0}$ to $M(\mu)$ described by Lemma \ref{short} is of length  \begin{align*}
    X' &= \sem + \emu \mu_{\emu} + \sum_{i=1}^{n-1-\emu} (n-i)(p-1).
\end{align*} Therefore if $X$ is the number of edges in the path from $\underline{0}$ to $\mu$ then \begin{align*}
    X &= X'+ (\mu_{\sem}-1)\sem+ \sum_{i=1}^{\sem-1} i \mu_i \\
    &\leq \sum_{i=1}^{n-1-\emu} (n-i)(p-1) + \emu \mu_{\emu} + \sum_{i=1}^{\sem} i \mu_i \\ &\leq \frac{1}{2}(p-1)(n^2-n).
\end{align*}

Suppose from now on that $\lm$ and $\mu$ are as in case (ii) so that $\el \leq \emu.$ Note that $\emu = 0$ if and only if $\el=0$ and therefore $\mu \neq St_p.$

If $\mu_{\emu} \neq 0$ then there is a path from $\lm$ to $\mu$ as follows. Starting with $\lm$ add $1$ to the first entry $\lm_0$ times where $\lm_0$ is the unique element of the set $\{0,\dots,p-2\}$ such that \begin{equation*}
    \lm_0 + \sum_{i=1}^{\emu}\lm_i \equiv \mu_{\emu} \mod p-1.
\end{equation*}
Then clear $1$'s repeatedly from the first nonzero entry until the entry in every position before $\emu$ in the resulting weight is $0$. By Lemma \ref{paths} this describes a path to a weight in which all entries before position $\emu$ are $0$ and the entry in position $\emu$ is $S_{\emu}=\mu_{\emu},$ where we adopt the notation $S_i$ from the proof of case (i). This process takes $\lm_0 + \sum_{i=1}^{\emu-1}S_i$ steps. Then add a $1$ and move it along to position $\sem.$ This describes a path from $\lm$ to the weight $M(\mu).$ The resulting path from $\lm$ to $M(\mu)$ and then to $\mu$ therefore takes $X$ steps where 
\begin{align*}
    X &=  \lm_0 + \sum_{i=1}^{\emu-1}S_i + \sum_{i=1}^{\emu-1} i \mu_i \\
    &\leq \frac{1}{2}(p-1)(n^2-n).
\end{align*} 

If $\mu_{\emu}=0$ and $\emu=n-1$ the our path from $\lm$ to $M(\mu)$ is as follows. Starting with $\lm$ add $1$ to the first entry $\lm_0$ times where $\lm_0$ is the unique element of the set $\{0,\dots,p-2\}$ such that \begin{equation*}
    \lm_0 + \sum_{i=1}^{\emu}\lm_i \equiv 1 \mod p-1.
\end{equation*}
Proceed as in the case above clearing $1$'s until all entries before position $\emu=n-1$ are $0$. Applying the arguments from the cases above the resulting weight will have a $1$ in position $n-1.$ We then clear $1$ from the entry in position $n-1$ resulting in the weight $\underline{0}$. Since $\lm_0 \leq p-2$ this path takes at most $(p-1)(n-1)$ steps. Finally we add $1$ and move it along into position $s_{\mu}$ resulting in the weight $M(\mu)$. Combining this with the path from $M(\mu)$ to $\mu$ described above we have a path from $\lm$ to $\mu$ that takes no more than  $(p-1)(n-1) +\sum_{i=1}^{n-2}(p-1)i = \frac{1}{2}(p-1)(n^2-n)$ steps.

Finally if $\mu_{\emu}=0$ and $\emu<n-1$ then our path from $\lm$ to $M(\mu)$ is as follows. Starting with $\lm$ add $1$ to the first entry $\lm_0$ times where $\lm_0$ is the unique element of the set $\{0,\dots,p-2\}$ such that \begin{equation*}
    \lm_0 + \sum_{i=1}^{\emu}\lm_i \equiv 0 \mod p-1.
\end{equation*}
Proceed as in the case above clearing $1$'s until all entries before position $\emu$ are $0$ resulting in a weight whose entry in position $\emu$ is $p-1$. Note that since $\el \leq \emu$ the entry in position $\emu+1$ of the resulting weight is $p-1.$ We then subtract $1$ from the entry in position $\emu$ and add it to the entry in position $\emu+1$ and repeat $p-1$ times. Finally add $1$ and move it into position $s_{\emu}$ resulting in the weight $M(\mu).$ This path required at most $(\emu+1)(p-1)$ steps. Since $\emu<n-1$, by combining this with the path from $M(\mu)$ to $\mu$ we have described a path from $\lm$ to $\mu$ that requires no more than $\frac{1}{2}(p-1)(n^2-n)$ steps. This completes the proof. 
\end{proof}

\bibliography{On_the_modular_McKay_graph}

\bibliographystyle{amsplain}

\noindent M.G. Norris, Department of Mathematics, King's College, London WC2R 2LS, UK \\
\textit{Email address: miriam.norris@kcl.ac.uk}

\end{document}